\documentclass[12pt,leqno]{amsart}
\usepackage{amsfonts,amsthm,amsmath,comment,xcolor,graphicx,url,tikz}

\theoremstyle{plain}
\newtheorem{thm}{Theorem}[section]
\newtheorem{prop}{Proposition}[section]
\newtheorem{lem}{Lemma}[section]
\newtheorem{cor}{Corollary}[section]

\theoremstyle{definition}
\newtheorem{df}{Definition}[section]
\newtheorem{rem}{Remark}[section]
\newtheorem{ex}{Example}[section]
\newtheorem{conj}{Conjecture}[section]
\newtheorem{prob}{Problem}[section]

\newcommand{\FF}{\mathbb{F}}

\newcommand{\RR}{\mathbb{R}}

\newcommand{\I}{\mathcal{I}}

\newcommand{\Z}{\mathbb{Z}}

\newcommand{\R}{\mathbb{R}}

\DeclareMathOperator{\Harm}{Harm}
\DeclareMathOperator{\Hom}{Hom}

\begin{document}

\title[Weighted T--G polynomials]{Weighted Tutte--Grothendieck polynomials of graphs}

\author[Chakraborty]{Himadri Shekhar Chakraborty*}
\address
	{
		Department of Mathematics, 
		Shahjalal University of Science and Technology\\ 
		Sylhet-3114, Bangladesh\\
	}
\email{himadri-mat@sust.edu}

\author[Miezaki]{Tsuyoshi Miezaki}
\address
	{
		Faculty of Science and Engineering, 
		Waseda University, 
		Tokyo 169-8555, Japan\\
	}
\email{miezaki@waseda.jp} 

\author[Zheng]{Chong Zheng}
\address
{
	Faculty of Science and Engineering, 
	Waseda University, 
	Tokyo 169-8555, Japan\\
}
\email{engzhengchong@ruri.waseda.jp}


\thanks{*Corresponding author}

\date{}
\maketitle

\begin{abstract}
In this paper, 
we introduce the notion of weighted chromatic polynomials of a graph
associated to a weight function~$f$ of a certain degree, and 
discuss some of its properties.
As a generalization of this concept, 
we define the
weighted Tutte--Grothendieck polynomials of graphs. 
When $f$ is harmonic, we notice that there is a 
correspondence between the weighted Tutte--Grothendieck polynomials of graphs 
and the weighted Tutte polynomials of matroids. 
Moreover, we present some constructions of the 
weighted Tutte--Grothendieck invariants for graphs as well as 
the weighted Tutte invariants for matroids.
Finally, we give a remark on the categorification of the 
weighted chromatic polynomials of graphs and weighted 
Tutte polynomials of matroids.
\end{abstract}

{\small

\vspace{-0.15in}

\quad
}


\section{Introduction}

Motivated by the concept of the harmonic weight enumerator of a binary code
introduced by Bachoc~\cite{Bachoc} and the relation between the weight 
enumerator of a code and the Tutte polynomial of a matroid pointed out by
Greene~\cite{Greene1976}, Chakraborty, Miezaki, and Oura~\cite{CMOxx}
gave the notion of the harmonic Tutte polynomial of 
a matroid, and presented the Greene type 
relation between the harmonic weight enumerator
of a code and the harmonic Tutte polynomial of a matroid.
Tutte polynomials have been studied by many authors. 
We refer the readers to~\cite{Tutte1947,Tutte1959}
for more discussions on Tutte polynomials of a matroid.
In the study of graph theory, the chromatic polynomial of a graph that
counts the number of all proper colouring of a graph with a given number
of colours is an invariant for graphs (see~\cite{Tutte1954,Tutte1967}). 

In this paper,
we introduce some new polynomials associate to a graph and a weight function~$f$ of 
a certain degree, specifically, the weighted chromatic polynomial 
and the weighted Tutte--Grothendieck polynomial.
We show that a weighted chromatic polynomial is a particular case
of a weighted Tutte--Grothendieck polynomial.
Moreover, we present a Greene type relation between the weighted 
Tutte--Grothendieck polynomials of graphs and 
the weighted Tutte polynomials of matroids, 
provided the weight function~$f$ is harmonic.
Furthermore, we present the concept of weighted Tutte--Grothendieck 
invariants of graphs as well as  weighted Tutte invariants of matroids,
and give a connection between them. 
Finally, we give a remark on the categorification of the weighted chromatic polynomials
of graphs and weighted Tutte polynomials of matroids.

This paper is organized as follows.
In Section~\ref{Sec:Preli}, we give the basic definitions and 
notations from graphs and matroids. We also give a brief discussion
about (discrete) harmonic functions.
In Section~\ref{Sec:WeightedChrom}, 
we define the weighted chromatic polynomials
of graphs and obtain a recurrence formula for it (Proposition~\ref{Prop:harmchrominduc}).
In Section~\ref{Sec:WeightedTutte},
we give a recursive definition of weighted Tutte polynomials (Proposition~\ref{Prop:TutteInduc}).
Moreover, if the weight function~$f$ is harmonic, 
we establish a relation between the weighted chromatic polynomials of graphs 
and the weighted Tutte polynomials of matroids (Theorem~\ref{Thm:RelationChromTutte}).
In Section~\ref{Sec:HarmTG},
we define the weighted Tutte--Grothendieck polynomial of a graph, and 
give a generalization of Theorem~\ref{Thm:RelationChromTutte} for this case (Theorem~\ref{Thm:HarmRecipe}).
In Section~\ref{Sec:WeightedInv},
we present the weighted Tutte--Grothendieck invariants for graphs (Theorem~\ref{Thm:PhiInv}) and the weighted Tutte invariants for matroids (Theorem~\ref{Thm:WeightedTutteInv}) and a correspondence between them (Theorem~\ref{Thm:GenHarmRecipe}). 
In Section 7,
we give categorifications to weighted chromatic polynomials (Theorem~\ref{Thm:ChromaticCate})
and weighted Tutte polynomials (Theorem~\ref{Thm:TutteCate}),
respectively.

All computer calculations in this study were performed with the aid of
Mathematica~\cite{Mathematica}.

\section{Preliminaries}\label{Sec:Preli}

In this section,
we give a brief discussion on graphs and matroids including 
the basic definitions and notations. 
We review~\cite{Aigner2007,Cameron,Oxley1992} for this discussion. 
We also recall~\cite{Bachoc,Delsarte} for the definition and properties of 
the (discrete) homogenous functions and (discrete) harmonic functions. 

\subsection{Graphs}

Let $G := (V,E)$ be a graph, where
$V$ denotes the set of vertices, 
and $E$ is the set of edges.
An edge is usually incident with two vertices.
But if the edge incident 
with equal end vertices, the edge is called a \emph{loop}.
A graph is called \emph{simple}
if it has neither loops nor multiple edges.
A \emph{path} in $G$ is a sequence of edges 
$(e_1, e_2, \ldots, e_{m - 1})$ 
having a sequence of vertices 
$(v_1, v_2,\ldots, v_n)$ 
satisfying $e_i \mapsto \{v_i, v_{i + 1}\}$ 
for $i = 1, 2,\ldots, m - 1$.
A \emph{circuit} in $G$ 
is a sequence of edges 
$(e_1, e_2, \ldots, e_{m - 1})$ 
having a sequence of vertices 
$(v_1, v_2,\ldots, v_n)$ 
satisfying $e_i \mapsto \{v_i, v_{i + 1}\}$ 
for $i = 1, 2,\ldots, m - 1$ and
$v_{1} = v_{n} = v$.
A circuit that does not repeat vertices is called \emph{cycle}. 
The \emph{connected component} of $G$ is a connected subgraph
of $G$ which is not a proper subgraph of another connected subgraph of $G$.
A \emph{bridge} of $G$ is an edge whose removal increase the number of 
connected components of $G$ by $1$. 
Throughout this note, 
we assume that the graphs are finite and not necessary to be simple.

\subsection{Matroids}

Let $E$ be a finite set of cardinality~$n$ and 
$2^{E}$ denotes the set of all subsets of $E$. 
A (finite) \emph{matroid} 
$ M $ is an ordered pair $ (E, \I) $,
where $ \I $ is the collection of subsets of $ E $ 
satisfying the following conditions:  
\begin{itemize}
	\item[(M1)] 
	$ \emptyset \in \I $.
	\item[(M2)] 
	$ I \in \I $ 
	and 
	$ J \subset I $ 
	implies 
	$ J \in \I $.
	\item[(M3)] 
	$ I, J \in \I $ 
	with $ |I| < |J| $ 
	then there exists 
	$ j \in J\setminus I $ 
	such that 
	$ I \cup \{ j \} \in \I $.
\end{itemize}

The elements of $ \I $ are called the \emph{independent} 
sets of $ M $, 
and $ E $ is called the \emph{ground set} of $ M $. 
A subset of the ground set $ E $ 
that is not belongs to $ \I $ is called \emph{dependent}. 
For example, 
let $G = (V,E)$ be a finite graph. 
Let $\I$ be the set of all subsets $A$ of $E$ 
for which the subgraph $(V,A)$ contains no cycle. 
Then $G$ is a matroid (See \cite{Cameron,Tutte1959}). 
Such a matroid is called \emph{graphic matroid}, and is denoted by $M_{G}$.

An independent set $I$ is called a \emph{maximal independent set} 
if it becomes dependent on adding any element of $ E \setminus I $. 
It follows from the axiom (M3) that the cardinality of all the 
maximal independent sets in a
matroid $M$ is equal, called the \emph{rank} of $M$. 
These maximal independent sets are called the \emph{bases} of $M$. 
The \emph{rank} $\rho(J)$ of an arbitrary subset $J$ of $E$ 
is the size of the largest independent subset of~$J$.
That is, 
$
\rho(J) 
:= 
\max_{I \subset J} 
\{ 
|I| : 
I \in \I 
\}.
$
This implies 
$ \rho $ 
maps 
$ 2^{E} $ 
into 
$\mathbb{Z}$. 
This function $ \rho $ is called the \emph{rank function} of $ M $. 
In particular, 
$ \rho(\emptyset) = 0 $.
We shall denote 
$\rho(E)$ the rank of~$M$.
We refer the readers to \cite{Oxley1992} for detailed discussion.

\subsection{Discrete homogeneous and harmonic functions}

Let~$\Omega = \{1,2,\ldots,n\}$ be a finite set.
We define  
$\Omega_{d} := \{ X \in 2^{\Omega} : |X| = d\}$
for $d = 0,1, \ldots, n$. 
We denote by 
$\R 2^{\Omega}$, $\R \Omega_{d}$
the real vector spaces spanned by the elements of  
$2^{\Omega}$, $\Omega_{d}$,
respectively. 
An element of 
$\R \Omega_{d}$
is denoted by
\begin{equation}\label{Equ:FunREd}
	f :=
	\sum_{Z \in \Omega_{d}}
	f(Z) Z
\end{equation}
and is identified with the real-valued function on 
$\Omega_{d}$
given by 
$Z \mapsto f(Z)$. 
Such an element 
$f \in \R \Omega_{d}$
can be extended to an element 
$\widetilde{f}\in \R 2^{\Omega}$
by setting, for all 
$X \in 2^{\Omega}$,
\begin{equation}\label{Equ:TildeF}
	\widetilde{f}(X)
	:=
	\sum_{Z\in \Omega_{d}, Z\subset X}
	f(Z).
\end{equation}
If an element 
$g \in \R 2^{\Omega}$
is equal to some 
$\widetilde{f}$, 
for 
$f \in \R \Omega_{d}$, 
we say that $g$ has degree $d$. 
We call the vector space $\RR\Omega_{d}$ 
the \emph{homogeneous space} of degree~$d$ and denote by $\Hom_{d}(n)$.
The differentiation $\gamma$ is the operator defined by the linear form 
\begin{equation}\label{Equ:Gamma}
	\gamma(Z) 
	:= 
	\sum_{Y\in {\Omega}_{d-1}, Y\subset Z} 
	Y
\end{equation}
for all 
$Z \in \Omega_{d}$
and for all $d=0,1, \ldots n$, 
and $\Harm_{d}(n)$ is the kernel of $\gamma$:
\[
	\Harm_{d}(n) 
	:= 
	\ker
	\left(
	\gamma\big|_{\RR \Omega_{d}}
	\right).
\]

\begin{rem}[\cite{Bachoc,Delsarte}]\label{Rem:Gamma}
	Let $f \in \Harm_{d}(n)$. 
	Then $\gamma^{d-i}(f) = 0$ for all $0 \leq i \leq d-1$.
	That is, from~(\ref{Equ:Gamma})
	$\sum_{\substack{Z \in \Omega_{d},\\ X \subset Z}} f(Z) = 0$
	for any $X \in \Omega_{i}$.
\end{rem}

Now we have the following technical lemma.

\begin{lem}[\cite{Bachoc}]\label{Lem:Bachoc}
	Let $f \in \Harm_{d}(n)$ and $J \subset \Omega$.
	Let
	\[
	f^{(i)}(J)
	:=
	\sum_{\substack{Z \in \Omega_{d}\\ |J \cap Z| = i}}
	f(Z).
	\]
	Then for all 
	$0 \leq i \leq d$,
	$f^{(i)}(J) = (-1)^{d-i} \binom{d}{i} \widetilde{f}(J)$.
\end{lem}

\begin{rem}\label{Rem:BachocLem}
	From the definition of $\widetilde{f}$ 
	for $f \in \Harm_{d}(n)$,
	we have $\widetilde{f}(J) = 0$
	for any $J \in 2^{\Omega}$ such that $|J| < d$. 
	Let $I, J \in 2^{\Omega}$ such that $I = \Omega\setminus J$.
	Then
	\begin{align*}
		\widetilde{f}(J)
		=
		\sum_{\substack{Z \in \Omega_{d}\\Z \subset J}}
		f(Z)
		=
		\sum_{\substack{Z \in \Omega_{d}\\ |Z \cap I|=0}}
		f(Z)
		=
		f^{(0)}(I)
		=
		(-1)^{d} \widetilde{f}(\Omega\setminus J).		
	\end{align*}	
	We have from the above equality that if $|J| > n-d$, 
	then $\widetilde{f}(J) = 0$. 
\end{rem}

For $X, Y \in 2^{\Omega}$, 
we introduce an operator $\circ$ on $\RR 2^{\Omega}$ as
\[
	\widetilde{f}(X) 
	\circ 
	\widetilde{f}(Y)
	:=
	\widetilde{f}(X \cup Y),
\]
which is associative, and distributive with respect to addition. 
Then we have the following remark.

\begin{rem}\label{Rem:Himadri}
	Let $I \subset \Omega$ and $J \subset \Omega\setminus I =: I^{c}$. 
	Then for $f \in \Harm_{d}(n)$ and by Remark~\ref{Rem:BachocLem}, 
	it is clear that
	\begin{itemize}
		\item [(i)]
		$\widetilde{f}(I)
		\circ
		\widetilde{f}(I^{c} \setminus J)
		=
		\widetilde{f}(I)
		\circ
		\widetilde{f}(\Omega \setminus (J \cup I))
		=
		(-1)^{d} 
		\widetilde{f}(J)$,
		
		\item [(ii)]
		$\widetilde{f}(I^{c} \setminus J)
		=
		\widetilde{f}(\Omega \setminus (J \cup I))
		=
		(-1)^{d} 
		\widetilde{f}(I)
		\circ
		\widetilde{f}(J)$.
	\end{itemize}
\end{rem}

\section{Weighted chromatic polynomials}\label{Sec:WeightedChrom}

In this section,
we discuss the weighted chromatic polynomial of a graph 
which is sometimes called the homogeneous chromatic polynomial of a graph.
For the definitions and notations of the classical chromatic polynomials of graphs
we refer the readers to~\cite{Aigner2007,Cameron}.

Let $G = (V,E)$ be a graph.
We assume that the number of edges is~$|E|=n$. 
Let $s$ be a bijective map from the edge set $E$ to $\Omega$, and 
$G(s)$ be a graph, where the edges are indexed by $s$.
We call $G(s)$ the \emph{labelled graph} and 
$s$ the \emph{label} of the graph~$G$.
For $A \subset E$, we denote by
$N_{A}(\lambda)$ the number of $\lambda$-colouring such that
the vertices adjacent to $A$ have the same colour.
Let $f$ be a homogeneous function of degree~$d$.
Then the \emph{weighted chromatic polynomial} of $G$ 
associated with $f$ and $s$ is defined as:
\[
	\chi_{f}(G(s);\lambda)
	:=
	\sum_{A\subset E}
	\widetilde{f}(s(E\setminus A))
	(-1)^{|A|}
	N_{A}(\lambda).
\]
Let $G_{A} := (V,A)$ be the subgraph of $G$ identified by $A$.
We denote by~$k(G_{A})$ the number of connected components of $G_{A}$.
Then $N_{A}(\lambda) = \lambda^{k(G_{A})}$. 
Therefore the weighted chromatic polynomial $\chi_{f}(G(s);\lambda)$
can be written as:
\[
	\chi_{f}(G(s);\lambda)
	=
	\sum_{A\subset E}
	\widetilde{f}(s(E\setminus A))
	(-1)^{|A|}
	\lambda^{k(G_{A})}.
\]

In the same manner as above, we can define the weighted chromatic polynomial of a graph 
related to any map $s$ instead of a bijective map. We will discuss the properties 
of such weighted chromatic polynomials in some subsequent papers~\cite{WTGII}. 
In this paper, 
we only consider $s$ as bijective.

If $f$ is a harmonic function of degree~$d$, 
we call $\chi_{f}(G(s);\lambda)$ the \emph{harmonic chromatic polynomial}
of $G$ associated with~$f$ and $s$. When $f = 1$, we get the classical chromatic
polynomial $\chi(G;\lambda)$ which is independent of the choice of $s$.

\begin{rem}\label{Rem:harmchromempty}
	For a graph $G$ with no edge and $m$ vertices, we have
	\[
		\chi_{f}(G(s);\lambda)
		=
		\widetilde{f}(\emptyset)
		\lambda^{m},
	\] 
	where $f \in \Hom_{d}(n)$ and $s$ is any label of $G$.
\end{rem}

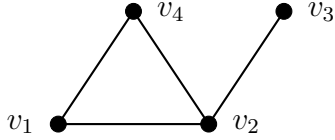
\begin{figure}[h]
	\begin{tikzpicture}
		\draw[fill=black] (0,0) circle (3pt);
		\draw[fill=black] (2,0) circle (3pt);
		\draw[fill=black] (3,1.5) circle (3pt);
		\draw[fill=black] (1,1.5) circle (3pt);
		\node at (-0.5,0) {$v_{1}$};
		\node at (2.5,0) {$v_{2}$};
		\node at (3.5,1.5) {$v_{3}$};
		\node at (1.5,1.5) {$v_{4}$};
		\draw[thick] (0,0) -- (2,0) -- (3,1.5);
		\draw[thick] (0,0)-- (1,1.5);
		\draw[thick] (2,0)-- (1,1.5);
	\end{tikzpicture}
	\caption{A graph with $4$ vertices and $4$ edges}
	\label{Fig:G44}
\end{figure}

\begin{ex}\label{Ex:WeightedChrom}
	Let the graph $G$ given in Figure~\ref{Fig:G44}.
	Let $\Hom_{1}(4) \ni f = a_{1} \{1\} + a_{2} \{2\} + a_{3} \{3\} + a_{4} \{4\}$.
	Let us consider the following label~$s$ for $G$:
	$\{v_{1},v_{2}\} \mapsto 4$, $\{v_{2},v_{3}\} \mapsto 1$, 
	$\{v_{2},v_{4}\} \mapsto 2$, $\{v_{1},v_{4}\} \mapsto 3$. 
	Then we have
	\begin{align*}
		\chi_{f}(G(s);\lambda)
		& =
		(a_{1}+a_{2}+a_{3}+a_{4}) 
		(\lambda^{4}
		-
		3 \lambda^{3}
		+
		3 \lambda^{2})
		-
		(a_{1}\lambda+a_{2}+a_{3}+a_{4})\lambda.
	\end{align*}
	For $f \in \Harm_{1}(4)$, we put $a_{1}+a_{2}+a_{3}+a_{4} = 0$.
\end{ex}

Let $G = (V,E)$ be a graph. For an edge $e$ of $G$, 
the \emph{deletion} $G\backslash e$ is the graph obtained by removing 
$e$ from the edge set $E$, and 
the \emph{contraction} $G/e$ 
is the graph obtained identifying end vertices of $e$ 
keeping all other adjacencies remain the same.

Now we give the following recurrence formula that can be used to calculate the weighted chromatic polynomials.

\begin{prop}\label{Prop:harmchrominduc}
	Let $G = (V,E)$ be a graph and $f \in \Hom_{d}(n)$.
	Suppose $e$ be an edge of $G$. 
	Then for a label~$s$ of $G$, we have 
	\[
		\chi_{f}(G(s);\lambda)
		=
		\widetilde{f}(s(e)) \circ \chi_{f}(G(s)\backslash e;\lambda)
		- 
		\chi_{f}(G(s)/e;\lambda).
	\]
\end{prop}

\begin{proof}
	Let $e$ be an edge of $G$.  
	Then for a subset $A$ of $E\setminus \{e\}$, 
	we denote by $t$ the term $(-1)^{|A|}\lambda^{k(G_{A})}$.
	Also every subset $A$ of $E\setminus \{e\}$ 
	corresponds to a pair of subsets
	$A$ and $A \cup \{e\}$ of $E$. Let $Q := A \cup \{e\}$. 
	Then we have $|E\setminus \{e\}| = |E| - 1$ and $|Q| = |A| + 1$.
	
	Suppose $e$ is a loop. 
	Then $k(G_{A}) = k(G_{Q})$.
	Clearly, $\widetilde{f}(s(E\setminus A))t$ and 
	$\widetilde{f}(s(E \setminus Q))t$ 
	be the terms in 
	$\chi_{f}(G(s)\setminus e;\lambda)$ 
	and
	$\chi_{f}(G(s)/e;\lambda)$, 
	respectively corresponding to the set $A$.
	Then the term in $\chi_{f}(G(s);\lambda)$ corresponding to $A$
	is $\widetilde{f}(s(e)) \circ \widetilde{f}(s(E\setminus A))t$,
	while the term corresponding to $Q$	is 
	$(-1)
	\widetilde{f}(s(E \setminus Q))t$.
	So
	\[
		\chi_{f}(G(s);\lambda)
		= 
		\widetilde{f}(s(e))  
		\circ
		\chi_{f}(G(s)\backslash e;\lambda)
		-
		\chi_{f}(G(s)/e;\lambda).
	\]
	
	Now suppose that $e$ is a bridge.
	Then $k(G_{A}) -1 = k(G_{Q})$.
	It is immediate that $\widetilde{f}(s(E\setminus A))t$ and 
	$\widetilde{f}(s(E\setminus Q))t/\lambda$ 
	be the terms in 
	$\chi_{f}(G(s)\backslash e;\lambda)$ and $\chi_{f}(G(s)/e;\lambda)$,
	respectively corresponding to the set $A$.
	Then the term in $\chi_{f}(G(s);\lambda)$ corresponding to $A$
	is $\widetilde{f}(s(e)) \circ \widetilde{f}(s(E\setminus A))t$,
	while the term corresponding to $Q$	is 
	$(-1)
	\widetilde{f}(s(E \setminus Q))t/\lambda$.
	So
	\[
		\chi_{f}(G(s);\lambda) 
		= 
		\widetilde{f}(s(e))
		\circ
		\chi_{f}(G(s)\backslash e;\lambda)
		-
		\chi_{f}(G(s)/e;\lambda).
	\]
	
	Finally, suppose $e$ is neither loop nor bridge.
	Then we can write
	\begin{multline*}
		\chi_{f}(G(s);\lambda) 
		= 
		\widetilde{f}(s(e))
		\circ
		\sum_{A \subseteq E \setminus e}
		(-1)^{|A|}\lambda^{k(G_{A})}
		+
		\sum_{A \subseteq E \setminus e}
		(-1)^{|Q|}
		\lambda^{k(G_{Q})}.		
	\end{multline*}
	Since $k(G_{A}) = k(G_{Q})$
	for $A \subseteq E \setminus e$, 
	we complete the proof.
\end{proof}

Now from the above proposition together with Remark~\ref{Rem:harmchromempty} 
we have the following inductive formulation:

\begin{prop}\label{Prop:HarmChromInduc}
	The weighted chromatic polynomials satisfies the following recursion:
	\begin{itemize}
		\item [(a)]
		If $G$ has $m$ vertices and no edge, then 
		$\chi_{f}(G(s);\lambda) = \widetilde{f}(\emptyset) \lambda^{m}$. 
		
		\item [(b)]
		If $e$ is a loop, then 
		$$\chi_{f}(G(s);\lambda) = \widetilde{f}(s(e)) \circ \chi_{f}(G(s)\backslash e;\lambda) - \chi_{f}(G(s)\backslash e;\lambda).$$
		
		
		\item [(c)]
		If $e$ is an edge which is not a loop, then
		$$\chi_{f}(G(s);\lambda) = \widetilde{f}(s(e)) \circ \chi_{f}(G(s)\backslash e;\lambda) - \chi_{f}(G(s)/e;\lambda).$$
	\end{itemize}
\end{prop}

\section{Weighted Tutte polynomials}\label{Sec:WeightedTutte}

In this section, 
we present a correspondence between
the harmonic chromatic polynomials of a graph associated with 
a harmonic function of degree~$d$ and 
the harmonic Tutte polynomials of a 
matroid associated with a harmonic function of degree~$d$.
We refer the reader to~\cite{Aigner2007, Cameron} for the basic
definitions and notations of the usual Tutte polynomial of a matroid.

Let $M = (E,\I)$ be a matroid with rank function $\rho$.
We assume that the cardinality of~$|E|=n$. 
Let $s$ be a bijective map from ground set $E$ to $\Omega$, and 
$M{(s)}$ be a matroid, where the elements of the ground set $E$ 
are indexed by $s$.
We call $M(s)$ the \emph{labelled matroid} 
and $s$ the \emph{label} of the matroid~$M$.
Let $f$ be a homogeneous function of degree~$d$. 
Then the \emph{weighted Tutte polynomial} of $M$ 
associated with~$f$ and $s$ is defined as follows: 
\[
	T_{f}(M(s);x,y)
	:=
	\sum_{A \subset E}
	\widetilde{f}(s(A))
	(x-1)^{\rho(E)-\rho(A)}
	(y-1)^{|A|-\rho(A)}.
\] 
In a graphic matroid, 
$s$ is a label for the matroid if and only if $s$ is a label for the underlying graph.

In the definition of the weighted Tutte polynomial of a matroid, 
we can consider any map $s$ instead of only a bijective map 
in a similar way as above. 
We will discuss the properties 
of such weighted Tutte polynomials in some subsequent papers~\cite{WTGII}. 
In this paper, we consider $s$ as a bijective mapping.

If $f$ is a harmonic function of degree $d$, 
then $T_{f}(M(s);x,y)$ is called the
\emph{harmonic Tutte polynomial} of $M$ associated to $f$ and $s$. 
If $f = 1$, then the weighted (harmonic) Tutte polynomials 
become the classical Tutte polynomial $T(M;x,y)$ independent of the label $s$.
Let $G = (V,E)$ be graph. 
Then for any $A \subset E$, we define
\[
	\rho(A) := |V|-k(G_{A}).
\]
Therefore, we find that $M_{G}$ is a matroid associated to $G$ 
having rank function~$\rho$. 

\begin{ex}\label{Ex:WeightedTutte}
	From Example~\ref{Ex:WeightedChrom}, 
	we have the weighted Tutte polynomial as below:
	\begin{multline*}
		T_{f}(M_{G}(s);x,y)
		=
		(a_1+a_2+a_3+a_4) 
		(x-1)^{2}
		+
		(a_2+a_3+a_4)
		(x-1)
		(y-1)\\
		+
		3(a_1+a_2+a_3+a_4)
		(x-1)
		+
		2(a_1+a_2+a_3+a_4) + a_1
	\end{multline*}
\end{ex}

\begin{rem}\label{Rem:rho_k}
	For $A \subset E$,
	$\rho(A) + k(G_{A}) = |V| = \rho(E) + k(G)$.
\end{rem}

Before giving a recurrence formula for the calculation of the
weighted Tutte polynomials, we recall~\cite{Cameron} for the definitions of 
the loop, coloop, deletion and contraction from the matroid theoretical point of view. 

Let $M = (E,\I)$ be a matroid. 
An element $e \in E$ is called a \emph{loop} if  $\{e\} \notin \I$,
or equivalently, if $\rho(\{e\}) = 0$. 
In a graphic matroid, $e$ is a loop if and only if it
is a loop of the underlying graph. 
An element $e \in E$ is a \emph{coloop} if 
it is contained in every basis of $M$.
This implies $\rho(A \cup \{e\}) = \rho(A)+1$,
whenever $e \notin A$. 
In a graphic matroid, $e$ is a coloop if and only if it is
a bridge. 
The \emph{deletion} of $e \in E$ is a matroid~$M\backslash e$ 
on the set $E \setminus \{e\}$
containing the independent sets of $M$ which are contained in $E \setminus \{e\}$. 
In a graphic matroid, deletion of $e$ corresponds to deletion of the edge $e$ from the
graph.
The \emph{contraction} of $e$ is the matroid
$M/e$ on the set $E \setminus \{e\}$ in which a set $A$ is independent 
if and only if $A \cup \{e\}$ is independent in $M$. 
In a graphic matroid, contraction of $e$ corresponds to 
contraction of the edge $e$.

\begin{prop}\label{Prop:TutteInduc}
	Let $M = (E,\I)$ be a matroid with rank function $\rho$ and $f \in \Hom_{d}(n)$.
	Suppose $e \in E$. Then for a label~$s$ of $M$, we have
	\begin{itemize}
		\item [(a)] $T_{f}(\emptyset;x,y) = \widetilde{f}(\emptyset)$, 
		where $\emptyset$ is the empty matroid.
		\item [(b)] If $e$ is a loop, then 
		$
			T_{f}(M(s);x,y) 
			= 
			T_{f}(M(s)\backslash e;x,y)
			+
			(y-1)\widetilde{f}(s(e)) 
			\circ 
			T_{f}(M(s)\backslash e;x,y).
		$
		\item [(c)] If $e$ is a coloop, then 
		$
			T_{f}(M(s);x,y) 
			= 
			(x-1)T_{f}(M(s)/e;x,y) 
			+ 
			\widetilde{f}(s(e)) 
			\circ 
			T_{f}(M(s)/e;x,y).
		$
		\item [(d)] If $e$ is neither a loop nor a coloop, then
		$$
			T_{f}(M(s);x,y) 
			= 
			T_{f}(M(s)\backslash e;x,y) 
			+ 
			\widetilde{f}(s(e)) \circ T_{f}(M(s)/e;x,y).
		$$
	\end{itemize}
\end{prop}

\begin{proof}
	For $e \in E$, 
	every subset $A$ of $E\setminus \{e\}$ 
	corresponds to a pair of subsets
	$A$ and $A \cup \{e\}$ of $E$.
	Let $\rho, \rho^{\prime}$ and $\rho^{\prime\prime}$ denote 
	the rank functions of $M, M\setminus e$ and $M/e$, respectively.
	Then (a) is trivial.
	Consider (b).
	Suppose $e$ is a loop. Then the following hold:
	\begin{itemize}
		\item [(i)] $|E\setminus \{e\}| = |E| - 1$;
		\item [(ii)] $\rho^{\prime}(E\setminus \{e\}) = \rho(E)$;
		\item [(iii)] $|A \cup \{e\}| = |A| + 1$;
		\item [(iv)] $\rho(A) = \rho(A \cup \{e\}) = \rho^{\prime}(A)$.
	\end{itemize}
	Let $\widetilde{f}(s(A))t$ and $\widetilde{f}(s(A\cup \{e\}))t$ 
	be the terms in 
	$T_{f}(M(s)\backslash e)$ and 
	$\widetilde{f}(s(e)) \circ T_{f}(M(s)\backslash e)$,
	respectively corresponding to the set $A$.
	Then the term in $T_{f}(M(s))$ corresponding to $A$
	is $\widetilde{f}(s(A))t$,
	while the term corresponding to $A \cup \{e\}$
	is $(y-1)\widetilde{f}(s(A \cup \{e\}))t$.
	So,
	\[
		T_{f}(M(s);x,y) 
		= 
		T_{f}(M(s)\backslash e;x,y)
		+
		(y-1) 
		\widetilde{f}(s(e)) 
		\circ 
		T_{f}(M(s)\backslash e; x,y).
	\]
	
	Now suppose that $e$ is a coloop.
	Then the following hold:
	\begin{itemize}
		\item [(i)] 
		$|E\setminus \{e\}| = |E| - 1$;
		\item [(ii)] 
		$\rho^{\prime\prime}(E\setminus \{e\}) = \rho(E)-1$;
		\item [(iii)] 
		$|A \cup \{e\}| = |A| + 1$;
		\item [(iv)] 
		$\rho(A) = \rho(A \cup \{e\})-1 = \rho^{\prime\prime}(A)$.
		\item [(v)] 
		$\rho(E) - \rho(A) = 
		\rho^{\prime\prime}(E\setminus\{e\}) -\rho^{\prime\prime}(A) + 1$.
		\item [(vi)] 
		$\rho(E) - \rho(A \cup \{e\}) 
		= \rho^{\prime\prime}(E\setminus\{e\}) -\rho^{\prime\prime}(A)$.
	\end{itemize}
	Let $\widetilde{f}(s(A))t$ and $\widetilde{f}(s(A\cup \{e\}))t$ 
	be the terms in 
	$T_{f}(M(s)/e)$ and 
	$\widetilde{f}(s(e)) \circ T_{f}(M(s)/e)$,
	respectively corresponding to the set $A$.
	Then the term in $T_{f}(M(s))$ corresponding to $A$
	is $(x-1)\widetilde{f}(s(A))t$,
	while the term corresponding to $A \cup \{e\}$
	is $f(s(A \cup \{e\}))t$.
	Therefore,
	\[
		T_{f}(M(s);x,y) 
		= 
		(x-1)T_{f}(M(s)/e;x,y) 
		+ 
		\widetilde{f}(s(e)) \circ T_{f}(M(s)/e;x,y).
	\]
	
	Finally, suppose $e$ is neither loop or coloop.
	Then we can write
	\begin{multline*}
		T_{f}(M(s);x,y) 
		= 
		\sum_{A \subseteq E \setminus \{e\}}
		\widetilde{f}(s(A))
		(x-1)^{\rho(E)-\rho(A)}
		(y-1)^{|A|-\rho(A)}\\
		+
		\sum_{A \subseteq E \setminus \{e\}}
		\widetilde{f}(s(A \cup \{e\}))
		(x-1)^{\rho(E)-\rho(A\cup \{e\})}
		(y-1)^{|A \cup \{e\}|-\rho(A\cup \{e\})}.		
	\end{multline*}
	Now for $A \subseteq E \setminus \{e\}$,
	\begin{align*}
		\rho(A\cup \{e\}) & = \rho^{\prime\prime}(A)+1 \text{ since $e$ is not a loop},\\
		\rho(E) & = \rho(M/e) + 1 \text{ since $e$ is not a loop},\\
		\rho(A) & = \rho^{\prime}(A) \text{ since $e$ is not a coloop},\\
		\rho(E) & = \rho(M\backslash e) \text{ since $e$ is not a coloop}.		
	\end{align*}
	This completes the proof.
\end{proof}

The following result shows a correspondence between the harmonic chromatic
polynomials and the harmonic Tutte polynomials.

\begin{thm}\label{Thm:RelationChromTutte}
	Let $G =(V,E)$ be a graph with $|V|=m$ and $|E| =n$. 
	Let $f \in \Harm_{d}(n)$.
	Then for any label $s$ of $G$, we have
	\[
		\chi_{f}(G(s);\lambda)
		=
		(-1)^{\rho(E)+d}
		\lambda^{k(G)}
		T_{f}(M_{G}(s);1-\lambda,0).
	\]
\end{thm}

\begin{proof}
	The matroid $M_{G}$ associated to the graph $G$ has the rank 
	{$\rho(J) = |V| - k(G_{J})$} for $J \subset E$.
	Let $\lambda$ be a positive integer.
	Since~$f\in \Harm_{d}(n)$, therefore we have from Remark~\ref{Rem:BachocLem} that
	$\widetilde{f}(E\setminus J) = (-1)^{d}\widetilde{f}(J)$ for any~$J\subset E$.
	Now we need to show by induction that 
	for~$f\in\Harm_{d}(n)$,
	\begin{equation}\label{Equ:Relation}
		\chi_{f}(G(s);\lambda)
		=
		(-1)^{\rho(E)+d}
		\lambda^{k(G)}
		T_{f}(M_{G}(s);1-\lambda,0).
	\end{equation}
	Let $\rho$, $\rho^{\prime}$ and $\rho^{\prime\prime}$ 
	be the rank functions of the matroids 
	$M_{G}$, $M_{G\backslash e}$ and~$M_{G/e}$, respectively.	
	
	The relation~(\ref{Equ:Relation}) is trivial when there is no edge since
	$\rho(E) = 0$, $k(G) = m$ and $T_{f}(M_{G}(s);1-\lambda,0) = \widetilde{f}(\emptyset)$. 
	
	Suppose $e$ is a loop. 
	Then $k(G) = k(G\backslash e)$.
	Also $\rho^{\prime}(E\setminus \{e\}) = \rho(E)$.
	Now by Proposition~\ref{Prop:TutteInduc}(b), we have
	\begin{multline}\label{Equ:TutteLoop}
		T_{f}(M_{G}(s);1-\lambda,0) \\
		= 
		T_{f}(M_{G\backslash e}(s);1-\lambda,0) 
		- 
		\widetilde{f}(s(e)) 
		\circ 
		T_{f}(M_{G\backslash e}(s);1-\lambda,0).
	\end{multline}
	Let $A \subset E\backslash \{e\}$.
	Then $T_{f}(M_{G\backslash e}(s);1-\lambda,0)$
	corresponds to the terms in 
	$T_{f}(M_{G}(s);1-\lambda,0)$ involving $A$,
	whereas 
	$\widetilde{f}(s(e)) \circ T_{f}(M_{G\backslash e}(s);1-\lambda,0)$
	corresponds to the terms in 
	$T_{f}(M_{G}(s);1-\lambda,0)$ involving $A \cup \{e\}$.
	On the other hand, 
	$\widetilde{f}(s(e)) \circ \chi_{f}(G(s)\backslash e;\lambda)$
	and
	$\chi_{f}(G(s)\backslash e;\lambda)$
	correspond to the term in 
	$\chi_{f}(G(s);\lambda)$
	involving~$A$ and~$A\cup \{e\}$,
	respectively. 
	Therefore, we can have
	\begin{equation}\label{Equ:RelationLoop1}
		\widetilde{f}(s(e)) 
		\circ
		\chi_{f}(G(s)\backslash e;\lambda)
		=
		(-1)^{\rho(E)+d}
		\lambda^{k(G)}
		T_{f}(M_{G\backslash e}(s);1-\lambda,0).
	\end{equation}
	\begin{equation}\label{Equ:RelationLoop2}
		\chi_{f}(G(s)\backslash e;\lambda)
		=
		(-1)^{\rho(E)+d}
		\lambda^{k(G)}
		(\widetilde{f}(s(e)) 
		\circ
		T_{f}(M_{G\backslash e}(s);1-\lambda,0)).
	\end{equation}
	Combining 
	(\ref{Equ:TutteLoop}), 
	(\ref{Equ:RelationLoop1}), 
	(\ref{Equ:RelationLoop2}), 
	and using Proposition~\ref{Prop:HarmChromInduc}(b)
	we can have the relation~(\ref{Equ:Relation}).
	
	Suppose $e$ is a bridge. 
	Then $k(G\backslash e) = k(G)+1$, and $k(G/e) = k(G)$.
	Then 
	$\rho^{\prime}(E\setminus \{e\}) 
	= 
	\rho^{\prime\prime}(E\setminus \{e\})
	= 
	\rho(E)-1$.
	By Proposition~\ref{Prop:TutteInduc}(c), we have
	\begin{multline}\label{Equ:TutteColoop}
		T_{f}(M_{G}(s);1-\lambda,0) \\
		= 
		(-\lambda)
		T_{f}(M_{G/e}(s);1-\lambda,0) 
		+ 
		\widetilde{f}(s(e)) 
		\circ 
		T_{f}(M_{G/e}(s);1-\lambda,0)
	\end{multline}
	Let $A \subset E\backslash \{e\}$.
	Then $(-\lambda) T_{f}(M_{G/e}(s);1-\lambda,0)$
	corresponds to the terms in 
	$T_{f}(M_{G}(s);1-\lambda,0)$ involving $A$,
	whereas 
	$\widetilde{f}(s(e)) \circ T_{f}(M_{G/e}(s);1-\lambda,0)$
	corresponds to the terms in 
	$T_{f}(M_{G}(s);1-\lambda,0)$ involving $A \cup \{e\}$.
	On the other hand, 
	$\widetilde{f}(s(e)) \circ \chi_{f}(G(s)\backslash e;\lambda)$
	and
	$\chi_{f}(G(s)/e;\lambda)$
	correspond to the term in 
	$\chi_{f}(G(s);\lambda)$
	involving~$A$ and~$A\cup \{e\}$,
	respectively.
	Therefore, we can have
	\begin{multline}\label{Equ:RelationBridge1}
		\widetilde{f}(s(e)) 
		\circ
		\chi_{f}(G(s)\backslash e;\lambda)\\
		=
		(-1)^{\rho(E)-1+d}
		\lambda^{k(G)+1}
		T_{f}(M_{G/e}(s);1-\lambda,0).
	\end{multline}
	\begin{multline}\label{Equ:RelationBridge2}
		\chi_{f}(G(s)/e;\lambda)\\
		=
		(-1)^{\rho(E)-1+d}
		\lambda^{k(G)}
		(\widetilde{f}(s(e)) 
		\circ
		T_{f}(M_{G/e}(s);1-\lambda,0)).
	\end{multline}
	Combining 
	(\ref{Equ:TutteColoop}), 
	(\ref{Equ:RelationBridge1}), 
	(\ref{Equ:RelationBridge2}) 
	and by using 
	Proposition~\ref{Prop:HarmChromInduc}(c), 
	we can have the relation~(\ref{Equ:Relation}).
	
	Finally, suppose that $e$ is neither a loop nor a bridge.
	Then $k(G) = k(G\backslash e) = k(G/e)$. 
	Also $\rho^{\prime}(E\setminus \{e\}) = \rho(E)$,
	and $\rho^{\prime\prime}(E\setminus \{e\}) = \rho(E)-1$. 
	Therefore, by similar argument as above we can show relation~(\ref{Equ:Relation})
	is true by using Proposition~\ref{Prop:HarmChromInduc}(c) together with Proposition~\ref{Prop:TutteInduc}(d).
	\end{proof}

	We can also give a direct proof of the above theorem with the 
	help of Remark~\ref{Rem:rho_k} as follows:
	
	\begin{align*}
		(-1)^{\rho(E)}
		& \lambda^{-k(G)}
		\chi_{f}(G(s);\lambda)\\
		& =
		\sum_{A \subset E}
		\widetilde{f}(s(E\setminus A))
		(-1)^{\rho(E)+|A|}
		\lambda^{k(G_{A})-k(G)}\\
		& =
		(-1)^{d}
		\sum_{A \subset E}
		\widetilde{f}(s(A))
		(-1)^{\rho(E)+|A|}
		\lambda^{\rho(E)-\rho(A)}\\
		& =
		(-1)^{d}
		\sum_{A \subset E}
		\widetilde{f}(s(A))
		(-1)^{\rho(E)+|A|}
		(-1)^{-2\rho(A)}
		\lambda^{\rho(E)-\rho(A)}\\
		&=
		(-1)^{d}
		\sum_{A \subset E}
		\widetilde{f}(s(A))
		(-1)^{\rho(E)-\rho(A)}
		(-1)^{|A|-\rho(A)}
		\lambda^{\rho(E)-\rho(A)}\\
		& =
		(-1)^{d}
		\sum_{A \subset E}
		\widetilde{f}(s(A))
		(-1)^{|A|-\rho(A)}
		(-\lambda)^{\rho(E)-\rho(A)}\\
		& =
		(-1)^{d}
		T_{f}(M_{G}(s);1-\lambda,0).
	\end{align*}
	Multiplying both sides by $(-1)^{\rho(E)} \lambda^{k(G)}$ we get the result.

\section{Harmonic Tutte--Grothendieck polynomials}\label{Sec:HarmTG}

In this section, we present the concept of harmonic Tutte--Grothendieck polynomial (harmonic T--G polynomial). 
Moreover, we give a generalization of the classical relation 
known as the recipe theorem between the Tutte--Grothendieck invariants 
and the Tutte polynomials 
to the case of harmonic T--G polynomials.

\begin{df}\label{Def:TGpoly}
Let $G = (V,E)$ be a graph with number of edges~$n$.
Let $s$ be a label of $G$,
and $e$ be an edge of $G$.
Let $f \in \Hom_{d}(n)$. 
Then the \emph{weighted Tutte-Grothendieck polynomial} 
of $G$ associated to $f$ and $s$ is denoted by 
$\Phi_{f}(G(s)) := \Phi_{f}(G(s);X,Y,\alpha,\beta)$ 
and defined recursively for $\alpha \neq 0$ and $\beta \neq 0$ 
as follows:
\begin{itemize}
	\item [(a)]
	If $G$ has no edge, then 
	$\Phi_{f}(G(s)) = \widetilde{f}(\emptyset)$.
	
	\item [(b)]
	If $e$ is a loop, then
	$$
	\Phi_{f}(G(s))
	= 
	\alpha
	\widetilde{f}(s(e)) \circ 
	\Phi_{f}(G(s)\backslash e)
	+
	(Y-\alpha)
	\Phi_{f}(G(s)\backslash e).
	$$
		
	\item [(c)]
	If $e$ is a bridge, then
	\[
		\Phi_{f}(G(s))
		= 
		X^\prime
		Y^{\prime}
		\widetilde{f}(s(e)) \circ 
		\Phi_{f}(G(s)\backslash e)
		+
		\beta
		Y^{\prime}
		\Phi_{f}(G(s)\backslash e),
	\]
	where 
	$X^\prime = X - \beta$ and 
	$Y^{\prime} := \dfrac{Y-\alpha}{\alpha}$.
	
	\item [(d)]
	If $e$ is neither a loop nor a bridge, then
	\[
		\Phi_{f}(G(s))
		= 
		\alpha 
		\widetilde{f}(s(e))
		\circ 
		\Phi_{f}(G(s)\backslash e)
		+
		\beta 
		\Phi_{f}(G(s)/e).
	\]
\end{itemize}
\end{df}
If $f$ is a harmonic function of degree~$d$, 
then we call $\Phi_{f}(G{(s)})$ the 
\emph{harmonic Tutte--Grothendieck polynomial} 
of $G$ associated to $f$ and $s$.

\begin{rem}\label{Rem:HarmTG}
	If $f \in \Harm_{d}(n)$, then by Remark~\ref{Rem:Himadri}, 
	we have instantly that
	$\Phi_{f}(G(s)) = (-1)^{d} P_{f}(G(s))$,
	where  $P_{f}(G(s)) := P_{f}(G(s);X,Y,\alpha,\beta)$ 
	is a polynomial of $G$ associated with $f$ and $s$
	satisfying the following recurrence for $\alpha \neq 0$ and $\beta \neq 0$:
	\begin{itemize}
		\item [(a)]
		If $G$ has no edge, then 
		$P_{f}(G(s)) = \widetilde{f}(\emptyset)$.
		
		\item [(b)]
		If $e$ is a loop, then
		$$
		P_{f}(G(s))
		= 
		\alpha P_{f}(G(s)\backslash e)
		+
		(Y-\alpha)
		\widetilde{f}(s(e)) 
		\circ 
		P_{f}(G(s)\backslash e).
		$$
		
		\item [(c)]
		If $e$ is a bridge, then
		\[
			P_{f}(G(s))
			= 
			X^\prime 
			Y^{\prime}
			P_{f}(G(s)\backslash e)
			+
			\beta
			Y^{\prime}
			\widetilde{f}(s(e)) 
			\circ 
			P_{f}(G(s)\backslash e),
		\]
		where 
		$X^{\prime} = X - \beta$ and 
		$Y^{\prime} = \dfrac{Y-\alpha}{\alpha}$.
		
		\item [(d)]
		If $e$ is neither a loop nor a bridge, then
		\[
			P_{f}(G(s))
			= 
			\alpha P_{f}(G(s)\backslash e)
			+
			\beta 
			\widetilde{f}(s(e))
			\circ 
			P_{f}(G(s)/e).
		\]
	\end{itemize}
\end{rem}

Now we give a generalization of the recipe theorem which presents a relation between
the harmonic T--G polynomial and the harmonic Tutte polynomial as follows.

\begin{thm}\label{Thm:HarmRecipe}
	Let $G = (V,E)$ be a graph with $|E| = n$, and
	$M_{G}$ be a matroid associated with $G$.
	Let $\Phi_{f}(G(s);X,Y,\alpha,\beta)$ be a harmonic T--G polynomial 
	of $G$ related to $f \in \Harm_{d}(n)$ and a label $s$ of $G$.
	Then for $\alpha \neq 0$ and $\beta \neq 0$, we have
	\begin{equation}\label{Equ:Recipe}
		\Phi_{f}(G(s); X,Y,\alpha, \beta)
		=
		(-1)^{d}
		\alpha^{|E|-|V|+k(G)}
		\beta^{|V|-k(G)}
		T_{f}
		\left(M_{G}(s);\frac{X}{\beta},\frac{Y}{\alpha}\right).		
	\end{equation}
\end{thm}

\begin{proof}
	From Remark~\ref{Rem:HarmTG}, it is sufficient to show that
	\begin{equation}\label{Equ:PGf}
		P_{f}(G(s); X,Y,\alpha, \beta)
		=
		\alpha^{|E|-|V|+k(G)}
		\beta^{|V|-k(G)}
		T_{f}
		\left(M_{G}(s);\frac{X}{\beta},\frac{Y}{\alpha}\right).
	\end{equation}
	If $G$ has no edge then the formula (\ref{Equ:PGf}) is trivial and so is (\ref{Equ:Recipe}).
	Let $e$ be a loop. Then $k(G) = k(G\backslash e)$.
	Now we have
	\begin{align*}
		P_{f}(G(s))
		& = 
		\alpha 
		P_{f}(G(s)\backslash e)
		+
		(Y-\alpha)
		\widetilde{f}(s(e)) 
		\circ
		P_{f}(G(s)\backslash e)\\
		& =
		\alpha.
		\alpha^{|E|-1-|V|+k(G)}
		\beta^{|V|-k(G)}
		T_{f}
		\left(M_{G\backslash e}(s);\frac{X}{\beta},\frac{Y}{\alpha}\right)\\
		& \quad +
		\alpha^{|E|-|V|+k(G)}
		\beta^{|V|-k(G)}
		\left(\frac{Y-\alpha}{\alpha}\right)
		\widetilde{f}(s(e)) 
		\circ
		T_{f}
		\left(M_{G\backslash e}(s);\frac{X}{\beta},\frac{Y}{\alpha}\right)\\
		& =
		\alpha^{|E|-|V|+k(G)}
		\beta^{|V|-k(G)}
		T_{f}
		\left(M_{G}(s);\frac{X}{\beta},\frac{Y}{\alpha}\right).
	\end{align*} 
	Hence (\ref{Equ:PGf}) as well as (\ref{Equ:Recipe}) is true when $e$ is loop.
	Let $e$ be a bridge. Then $k(G)+1 = k(G\backslash e)$ and we have
	\begin{align*}
		P_{f}(G(s))
		& = 
		(X-\beta)
		\left(\frac{Y-\alpha}{\alpha}\right)
		P_{f}(G(s)\backslash e)
		+
		\beta\left(\frac{Y-\alpha}{\alpha}\right)
		\widetilde{f}(s(e)) 
		\circ 
		P_{f}(G(s)\backslash e)\\
		& =
		\alpha^{|E|-|V|+k(G)}
		\beta^{|V|-k(G)}
		\left(\frac{X-\beta}{\beta}\right)
		\left(\frac{Y-\alpha}{\alpha}\right)
		T_{f}
		\left(M_{G\backslash e}(s);\frac{X}{\beta},\frac{Y}{\alpha}\right)\\
		& \quad +
		\alpha^{|E|-|V|+k(G)}
		\beta.\beta^{|V|-k(G)-1}
		\left(\frac{Y-\alpha}{\alpha}\right)
		\widetilde{f}(s(e)) 
		\circ
		T_{f}
		\left(M_{G\backslash e}(s);\frac{X}{\beta},\frac{Y}{\alpha}\right)\\
		& =
		\alpha^{|E|-|V|+k(G)}
		\beta^{|V|-k(G)}
		\left(\frac{X-\beta}{\beta}\right)
		T_{f}
		\left(M_{G/e}(s);\frac{X}{\beta},\frac{Y}{\alpha}\right)\\
		& \quad +
		\alpha^{|E|-|V|+k(G)}
		\beta^{|V|-k(G)}
		\widetilde{f}(e) \circ
		T_{f}
		\left(M_{G/e}(s);\frac{X}{\beta},\frac{Y}{\alpha}\right)\\
		& =
		\alpha^{|E|-|V|+k(G)}
		\beta^{|V|-k(G)}
		T_{f}
		\left(M_{G}(s);\frac{X}{\beta},\frac{Y}{\alpha}\right).
	\end{align*} 
	So, (\ref{Equ:Recipe}) is true when $e$ is a bridge. 
	Now let $e$ be neither a loop nor a bridge. 
	Then $k(G) = k(G\backslash e) = k(G/e)$
	and similarly as above we find (\ref{Equ:Recipe}) is true.
\end{proof}

\begin{ex}\label{Ex:WCP}
	If we take 
	$\Phi_{f}(G(s)) = \dfrac{\chi_{f}(G(s))}{\lambda^{k(G)}}$
	with $f \in \Hom_{d}(n)$, $X = \lambda -1$, $Y = 0$, $\alpha = 1$ and $\beta = -1$, 
	we immediately get the weighted chromatic polynomial.
	Moreover, Theorem~\ref{Thm:HarmRecipe} implies
	Theorem~\ref{Thm:RelationChromTutte} when $f \in \Harm_{d}(n)$.
\end{ex}

\section{Weighted invariants for graphs and matroids}\label{Sec:WeightedInv}

Tutte~\cite{Tutte1947} pointed out an idea of graph invariant which is known as 
Tutte--Grothendieck invariant (T--G invariant) for all functions that satisfy 
deletion-contraction recurrence was thoroughly developed and generalized 
by Brylawski~\cite{Brylawski}. 
Motivated by the concept of T--G invariant, in this section, 
we introduce the notion of weighted T--G invariant 
which is  also called harmonic T--G invariant in some 
particular cases. 

\subsection*{Weighted T--G invariants}

Let $G=(V,E)$ be a graph with $|E| = n$.
Let $s$ be a label of $G$.
We denote the set of all labels~$s$ of~$G$ by $S(G)$.
It is easy to compute that $\sharp S(G) = n!$. 
Now we define
\begin{equation}\label{Equ:PhiHat}
	\widehat\Phi_{f}(G)
	:=
	\sum_{s \in S(G)}
	\Phi_{f}(G(s)).
\end{equation}

\begin{rem}
	For $f = 1$, 
	$\widehat{\Phi}_{f}(G) = n! \Phi(G)$,
	where $\Phi(G)$ is the classical T--G invariant. 
\end{rem}

We denote by $S_{n}$ the symmetric group of order~$n$. 
Then the following result is immediate from the construction of $\widehat{\Phi}_{f}(G)$. 

\begin{thm}\label{Thm:PhiInv}
	Let $G = (V,E)$ be a graph with $|E| = n$. Then for any $f\in \Hom_d(n)$, 
	we have 
	\begin{itemize}
		\item [(i)]
		$\widehat{\Phi}_{f}(G)$ is an invariant for graphs. 
		
		\item [(ii)]
		Let $R(f) := \sum_{\sigma \in S_{n}} \sigma f$.
		Then $\Phi_{R(f)}(G)$ is an invariant for graphs,
		where $\Phi_{R(f)}(G)$ denotes the weighted T--G polynomial of $G$ 
		associated to $R(f)$ and independent of the choice of $s$.
		Moreover, $\widehat{\Phi}_{f}(G) = \Phi_{R(f)}(G)$.

		\item  [(iii)]
		Let $\sigma f = f$ for all $\sigma \in S_{n}$. 
		Then $\Phi_{f}(G)$ is an invariant for graphs. 
		Also, $\widehat{\Phi}_{f}(G) = n! \Phi_{f}(G)$.
	\end{itemize}
\end{thm}

We call $\widehat{\Phi}_{f}(G)$ the 
\emph{weighted Tutte--Grothendieck invariant} for graphs.
If $f \in \Harm_{d}(n)$, we call $\widehat{\Phi}_{f}(G)$ the 
\emph{harmonic Tutte--Grothendieck invariant} for graphs,
which is zero for $d \neq 0$.

\subsection*{Weighted Tutte invariants}

Let $M$ be a matroid with ground set $E$ of cardinality $n$.
Let $s$ be a label of $M$.
Define
\[
	\widehat{T}_{f}(M;x,y)
	:=
	\sum_{s\in S(M)} 
	T_{f}(M(s);x,y),
\]
where $S(M)$ interprets a similar meaning as $S(G)$ for the case of graphs.

\begin{rem}
	For $f = 1$, 
	$\widehat{T}_{f}(M;x,y) = n!T(M;x,y)$. 
\end{rem}

\begin{thm}\label{Thm:WeightedTutteInv}
	Let $M = (E,\I)$ be a matroid with $|E| = n$. Then for any $f\in \Hom_d(n)$, 
	we have 
	\begin{itemize}
		\item [(i)]
		$\widehat{T}_{f}(M)$ is an invariant for graphs. 
		
		\item [(ii)]
		Let $R(f) := \sum_{\sigma \in S_{n}} \sigma f$.
		Then $T_{R(f)}(M)$ is an invariant for matroids,
		where $T_{R(f)}(M)$ denotes the weighted Tutte polynomial of $M$ 
		associated to $R(f)$ and independent of the choice of label $s$.
		Moreover, $\widehat{T}_{f}(M) = T_{R(f)}(M)$.
		
		\item  [(iii)]
		Let $\sigma f = f$ for all $\sigma \in S_{n}$. 
		Then $T_{f}(M)$ is an invariant for matroids. 
		Also, $\widehat{T}_{f}(M) = n! T_{f}(M)$.
	\end{itemize}
\end{thm}

We call $\widehat{T}_{f}(M;x,y)$ the 
\emph{weighted Tutte invariant} for matroids.
If $f \in \Harm_{d}(n)$, we call $\widehat{T}_{f}(M;x,y)$ the 
\emph{harmonic Tutte invariant} for matroids which is zero for $d \neq 0$.

\begin{ex}
	Let $M_{1}$ be a vector matroid over $\FF_{2}$ 
	and $M_{2}$ be a vector matroid over $\FF_{3}$
	as given below:
	\[
		M_{1}
		=
		\begin{pmatrix}
			1 & 1 & 1 & 1 & 0 & 0 & 0\\
			1 & 1 & 0 & 0 & 1 & 1 & 0\\
			0 & 0 & 1 & 0 & 1 & 0 & 1
		\end{pmatrix},
		\quad
		M_{2}
		=
		\begin{pmatrix}
			2 & 1 & 0 & 1 & 0 & 1 & 2\\
			1 & 1 & 1 & 0 & 0 & 0 & 0\\
			0 & 0 & 0 & 0 & 1 & 1 & 2
		\end{pmatrix}.
	\]
We have checked numerically that 
for all $f\in \Hom_{d}(7)\ (1\leq d\leq 7)$, 
$\widehat{T}_{f}(M_{1},x,y) = \widehat{T}_{f}(M_{2},x,y)$. 
For the construction of vector matroids, we refer the readers to~\cite{Cameron}.

\end{ex}

\begin{prob}
Find a non-isomorphic matroid pair $(M_1, M_2)$ 
having the same Tutte polynomial but the 
different weighted Tutte invariants. 

\end{prob}

The following result gives a correspondence between
the weighted T--G invariant and the weighted Tutte invariant 
which is an analogue to Theorem~\ref{Thm:HarmRecipe}.

\begin{thm}\label{Thm:GenHarmRecipe}
	Let $\widehat\Phi_{f}$ be a weighted T--G invariant associated to 
	$f \in \Hom_{d}(n)$. Then for all graphs $G = (V,E)$ 
	we have
	\begin{equation}\label{Equ:TGInvTutte}
		\widehat\Phi_{f}(G)
		=
		\alpha^{|E|-|V|+k(G)}
		\beta^{|V|-k(G)}
		\widehat{T}_{f}
		\left(M_{G};\frac{X}{\beta},\frac{Y}{\alpha}\right),	
	\end{equation}
	where $X,Y,\alpha,\beta$ interpret the same meaning as in 
	weighted T--G polynomials associated to $f$.
\end{thm}

\begin{proof}
	Let $A$ be any subset of $E$.
	Let $\widetilde{f}(s(A))t$ 	be the term in 
	$T_{f}(M_{G}(s))$ corresponding to the set $A$ and label $s$ of $M_{G}$.
	Since $s$ is a bijective map from $E$ to $\Omega$, 
	therefore, 
	$\sum_{s \in S(M_{G})} \widetilde{f}(s(A))t = \sum_{s \in S(M_{G})} \widetilde{f}(s(E\setminus A))t$. 
	This fact together with Proposition~\ref{Prop:TutteInduc} 
	and Definition~\ref{Def:TGpoly} implies~{(\ref{Equ:TGInvTutte})}.
\end{proof}

\subsection*{Weighted chromatic invariant}

Let $s$ be a label of a graph $G$ with $n$ edges. Then from Example~\ref{Ex:WCP}, 
if we take $\Phi_{f}(G(s)) = \dfrac{\chi_{f}(G(s))}{\lambda^{k(G)}}$
with $f \in \Hom_{d}(n)$, $X = \lambda -1$, $Y = 0$, $\alpha = 1$ and $\beta = -1$, 
we get from definition~(\ref{Equ:PhiHat}) that
\[
	\widehat{\chi}_{f}(G(s))
	=
	\sum_{s\in S(G)} 
	\chi_{f}(G(s)),
\]
which is from Theorem~\ref{Thm:PhiInv}, an invariant for graphs.
We call this invariant the \emph{weighted chromatic invariant}
for graphs.
When $f \in \Harm_{d}(n)$,
we call this invariant the 
\emph{harmonic chromatic invariant}
for graphs which is zero for $d \neq 0$.
Moreover, Theorem~\ref{Thm:GenHarmRecipe} 
yields therefore the following corollary.

\begin{cor}
	We have
	\[
		\widehat\chi_{f}(G;\lambda)
		=
		(-1)^{\rho(E)}
		\lambda^{k(G)}
		\widehat{T}_{f}(M_{G};1-\lambda,0).
	\]
\end{cor}

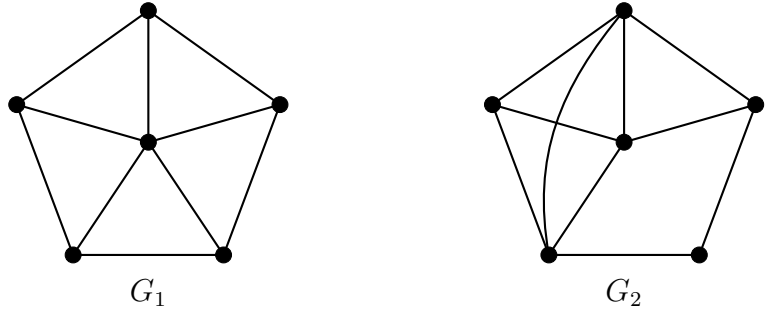
\begin{figure}[h]
	\begin{picture}(300,120)
		\put(10,10){
		\begin{tikzpicture}
			\draw[fill=black] (0,0) circle (3pt);
			\draw[fill=black] (2,0) circle (3pt);
			\draw[fill=black] (-0.75,2) circle (3pt);
			\draw[fill=black] (2.75,2) circle (3pt);
			\draw[fill=black] (1,3.25) circle (3pt);
			\draw[fill=black] (1,1.5) circle (3pt);
			\draw[thick] (0,0) -- (2,0) -- (2.75,2) -- (1,3.25) -- (-0.75,2) -- (0,0);
			\draw[thick] (0,0)-- (1,1.5);
			\draw[thick] (-0.75,2)-- (1,1.5);
			\draw[thick] (2,0)-- (1,1.5);
			\draw[thick] (2.75,2)-- (1,1.5);
			\draw[thick] (1,3.25)-- (1,1.5);
			\node at (1,-0.5) {$G_{1}$};
		\end{tikzpicture}}
		\put(190,10){
			\begin{tikzpicture}
				\draw[fill=black] (0,0) circle (3pt);
				\draw[fill=black] (2,0) circle (3pt);
				\draw[fill=black] (-0.75,2) circle (3pt);
				\draw[fill=black] (2.75,2) circle (3pt);
				\draw[fill=black] (1,3.25) circle (3pt);
				\draw[fill=black] (1,1.5) circle (3pt);
				\draw[thick] (0,0) -- (2,0) -- (2.75,2) -- (1,3.25) -- (-0.75,2) -- (0,0);
				\draw[thick] (0,0)-- (1,1.5);
				\draw[thick] (-0.75,2)-- (1,1.5);
				\draw[thick] (0,0) to[out=100, in=230] (1,3.25);
				\draw[thick] (2.75,2)-- (1,1.5);
				\draw[thick] (1,3.25)-- (1,1.5);
				\node at (1,-0.5) {$G_{2}$};
		\end{tikzpicture}}
	\end{picture}
\caption{Two non-isomorphic graphs}
\label{Fig:NonIsoGraph}
\end{figure}

\begin{ex}
	Let $G_{1}$ and $G_{2}$ be two non-isomorphic graphs 
	shown in Figure~\ref{Fig:NonIsoGraph} having the same 
	classical chromatic polynomials.
	Then we have $\Omega = \{1,2,\ldots,10\}$.
	Let $f=\sum_{X\in \Omega_{4}} X$. 
	Then the weighted chromatic invariant for $G_1$ is
	\[
		\widehat{\chi}_{f}(G_{1};\lambda)
		=
		10! 
		\times 
		( 210 \lambda^6 - 1260 \lambda^5 + 2975 \lambda^4 - 3450 \lambda^3 + 1960 \lambda^2 - 435 \lambda),
	\]
	whereas the weighted chromatic invariant for $G_2$ is
	\[
		\widehat{\chi}_{f}(G_{2};\lambda)
		=
		10! 
		\times 
		(210 \lambda^6 - 1260 \lambda^5
		  + 2975 \lambda^4 - 3434 \lambda^3 + 1925 \lambda^2 - 416 \lambda).
	\]
	Therefore, 
	$\widehat{\chi}_{f}(G_{1};\lambda) \neq \widehat{\chi}_{f}(G_{2};\lambda)$. 
	This concludes that the weighted chromatic invariants is stronger than 
	the classical chromatic polynomials.
\end{ex}

Now we recall~\cite{JuneHah} for the concept of log-concave sequences.
A sequence of real numbers
$a_0, a_1,...,a_n$  
is said to be \emph{log-concave} if 
$ a_{i-1} a_{i+1} \leq a_{i}^2$
for all
$0 < i < n$. 
Observing the above example, 
we may conclude the following conjecture.

\begin{conj}
	Let 
	$\widehat{\chi}_{f}(G;\lambda) = a_{n} \lambda^{n} - a_{n-1} \lambda^{n-1} + \cdots + (-1)^{n} a_{0}$ be a weighted chromatic invariant of a graph~$G$ associated to $f \in \Hom_{d}(n)$ such that $f$ is symmetric.
	Then the sequence $a_{0},a_{1},\ldots,a_{n}$ is log-concave. 
\end{conj}


\section{Homological categorifications of weighted polynomials}
Originated from category theory,
a \textit{categorification} in homology theory usually means 
to categorify 
a quantity or a polynomial 
with some homology groups.
For example, Khovanov homology 
is a categorification of 
Jone's polynomial,
Floer homology is 
a categorification of 
Alexander-Conway polynomial,
and
Magnitude homology is a categorification of
magnitude. 

In this section, we apply the method in \cite{categorification1}
and \cite{categorification2}
to construct  chain complexes and give 
categorifications to the 
weighted chromatic polynomial 
and weighted Tutte polynomial,
respectively. 
We omit some basic facts of algebraic topology
and homological algebra.
For  details, the readers may refer \cite{Munkres}
and \cite{Rotman}.

We categorify the 
weighted polynomial
by assigning 
a real-valued  weight 
to the specially constructed chain complexes,
and show the relationship between the chain complexes and
our invariant in Section~\ref{Sec:WeightedInv}.
We will write another paper
to discuss the general weighted chain complex,
and discuss only special cases in this paper.

\subsection{A categorification of weighted chromatic polynomials}

\label{Sec:CateHarmChrom}

Let $\mathcal{M}$ be a  free $\mathbb{Z}$-module with
$\mathcal{M}= M_0 \oplus M_1$, 
where  $M_0$ and $M_1$ are subgroups of $\mathcal{M}$
and are generated by basis element $1$
and $x$, respectively.
Hence $\mathcal{M}$ can be considered as 
 a graded free $\mathbb{Z}$-module
with non-trivial elements only at 
degree $0$ and $1$.

\begin{df}
Let $G(s)=(V, E)$ be a labelled graph 
with $n$ edges.
We use a vector 
$\epsilon=(\epsilon_1, \epsilon_2, \ldots, \epsilon_n) \in \{0,1\}^n$
 to
represent any given
 subgraph $G(s)'$ of $G(s)$ whose vertices 
are $V$.
Precisely, 
for $i=1, 2, \ldots, n,$ 
if the $i$-th edge exists in $G(s)'$, 
then we let $\epsilon_i= 1$;
and $\epsilon_i= 0$ 
otherwise.
We call  $\epsilon$ the 
\textit{edge vector } of $G(s)$.
Given any edge vector, we
also denote the corresponding subgraph
$G_\epsilon.$
We denote by $l(\epsilon)$ the number of $1$s in edge vector 
$\epsilon.$
Let $\mathcal{E}^a (G(s))$ be the set of all 
edge vectors $\epsilon$ with $l(\epsilon)=a$,
$a\in \mathbb{N}$.
\end{df}

It is easy to check that 
Definition $7.1$ is well-defined, 
and the set of all subgraphs of $G$ whose vertices are $V$ is
on 1-1 correspondence to the set of all edge vectors.

In order to construct a cochain complex with $M$,
we first define the cochain group as follows.
\begin{df}
The \textit{$q$-th chromatic cochain group} of a graph $G$ is defined as 
$$C^q(G(s)):= \bigoplus_{l(\epsilon)=q}  M^{\otimes k(G_\epsilon)},$$
where $\otimes$ denotes tensor product.
\end{df}

We also write 
$$C^q(G(s))= \bigoplus_{j\geq 0} C^{q,j}(G(s)),$$
where 
$C^{q,j}(G(s))$ denotes
the elements of degree $j$ of $C^q(G(s)).$ 
Hence $C^q(G(s)) $ can also be considered
as a bigraded module.

Let $f: 2^{\Omega} \to \mathbb{R}$
be a discrete function of degree~$d$.
Since each 
$M^{\otimes k(G_\epsilon)}$
 corresponds to 
an edge vector $\epsilon$, 
by the definition,
 it also corresponds to 
 an $f$-value. 
Therefore, 
besides the dimension of the graded module, 
we can assign a weighted dimension to 
the graded module, where the weight is given 
by $f$-values.
We begin with the definition of weighted rank as follows.

\begin{df}
Let $G$ be a $\mathbb{Z}$-module
with decomposition 
$G= \bigoplus_i G_i$.
We suppose that for each $G_i$, there is a corresponding 
$f$-value.
Then we define the \textit{$f$-\rm{rank}} of $G$ as  
$$f\mathrm{rank}(G):= \sum_i f(G_i) \cdot \dim_{\mathbb{Q}}(G_i \otimes \mathbb{Q}).$$

\end{df}

Note that $\dim_{\mathbb{Q}}(G_i \otimes \mathbb{Q})$
is known as the 
\textit{rank} of $G_i$.

\begin{df}
Let $\mathcal{N}= \bigoplus_{i} {N_i}$ be a 
graded $\mathbb{Z}$-module
with an $f$-value $f(\mathcal{N})$,
where ${N_i}$ denotes the
set of homogeneous elements of degree $i$.
The \textit{graded dimension} of $\mathcal{N}$
is defined as 
$$q\mathrm{dim}(\mathcal{N}):= \sum_{i} q^i \cdot \mathrm{rank}(N_i),$$
and the \textit{$f$-value graded dimension 
of $\mathcal{N}$}
is defined as
$$fq\mathrm{dim}(\mathcal{N}):= \sum_{i} q^i \cdot f\mathrm{rank}(N_i).$$

\end{df}

\begin{lem}
Let $\mathcal{N}$, $\mathcal{G}$, and $\mathcal{H}$
be graded modules 
with $f$-values and 
$\mathcal{N}= \mathcal{G} \oplus \mathcal{H}$.
Then 
$$ fq\mathrm{dim}(\mathcal{N}) = fq\mathrm{dim}(\mathcal{G}) + fq\mathrm{dim}(\mathcal{H}).$$

\end{lem}

\begin{proof}
\begin{equation*}
\begin{split}
fq\mathrm{dim}(\mathcal{N})&= \sum_{i} q^i \cdot f\mathrm{rank}(N_i) \\
							&= \sum_{i} q^i \cdot f\mathrm{rank}(G_i \oplus H_i) \\
&= \sum_{i} q^i \cdot f(G_i)\cdot \mathrm{rank}(G_i) + f(H_i)\cdot \mathrm{rank}(H_i)) \\
&= \sum_{i} q^i \cdot f(G_i)\cdot \mathrm{rank}(G_i) +
\sum_{i} q^i \cdot f(H_i)\cdot \mathrm{rank}(H_i)		\\
&=fq\mathrm{dim}(\mathcal{G})	+ fq\mathrm{dim}(\mathcal{H}), \\
\end{split}							
\end{equation*}
where $\mathrm{rank}(G_i)= \dim_{\mathbb{Q}}(G_i \otimes \mathbb{Q}).$
\end{proof}

According to
 the above lemma,
 we can calculate as an example
that 
for two graded module $\mathcal{M}$ defined  above,
suppose that they correspond to $f$-values $a_1$ and
$a_2$,
then
$fq\mathrm{dim}(\mathcal{M}\oplus \mathcal{M})= (a_1+ a_2)(1+q)$,
and
$fq\mathrm{dim}(\mathcal{M}^{\otimes m}\oplus \mathcal{M}^{\otimes n})= a_1(1+q)^m
+a_2(1+q)^n$.

The desired chain complex 
is constructed by the cochain group
above and a \textit{coboundary operator}.
A coboundary operator 
is a 
module homomorphism 
between two cochain groups
whose dimensions differ by $1$.
By the correspondence between cochain 
groups and edge vectors,
the coboundary operator is induced by a 
map that increases exactly one edge.

We define the coboundary operator 
by defining a map for its each element module,
and then linearly extend them.
Let
$d_{\epsilon}: \mathcal{E}^a(G(s)) \to   \mathcal{E}^{a+1}(G(s)) $
be a map 
that is identity on all elements but 
changes exact one $0$ to $1$.
For an edge $\tilde{e}$, if adding the edge
 does not change the
number of connected components,
then we define the element module 
map
$\tilde{d_{\epsilon}}: M^{\otimes k} \to M^{\otimes k}$
to be 
the identity map.
On the other hand, 
if the added edge $\tilde{e}$   
joins two connected components,
that is, the number of connected components decreases
by one, 
then 
we define the element module 
map
$\tilde{d_{\epsilon}}: M^{\otimes k} \to M^{\otimes k-1}$
to be 
identity on the tensor elements 
which are 
irrelevant to this process.
Moreover, for the two  components 
 that are connected,
 the map  is defined to be 
$M \otimes M \to M$
with 
$m(1\otimes 1)=1 $, 
$m(1\otimes x)=m(x\otimes 1)= x$,
and 
$m(x\otimes x)=0$.

\begin{df}[coboundary]

The coboundary operator 
$d^q: C^q(G(s)) \to C^{q+1}(G(s))$
 is defined as
$$d^q:= \sum_{l(\epsilon)=q} (-1)^ {n(\epsilon)} \tilde{d}_{\epsilon},$$
where 
$ n(\epsilon)=\sum_{i=1}^{k_0} \epsilon_i $,
and $k_0$ is the index of the new-added edge in 
the given order.
\end{df}

By definition and simple calculation,
 it is uncomplicated to verify the following statements.
\begin{lem}
\label{lemma7d}
The following statements are true.
\begin{enumerate}

\item $d$ is degree preserving.

\item $d\circ d=0.$
\end{enumerate}
\end{lem}

(2) of Lemma \ref{lemma7d} implies that 
$(C^{\star}, d^{\star})$
forms a cochain complex.
We denote by $H^{\star}(G):= \mathrm{ker}(d^{q})/\mathrm{im}(d^{q-1}),$
the corresponding cohomology,
where 
$\mathrm{ker}(d^{q})$ denotes the kernel of $d^{q}$, 
and 
$\mathrm{im}(d^{q-1}) $  
denotes the image of $d^{q-1}$.
It is shown in \cite{categorification1}
that, 
the  cochain complex does not 
depend on the choice of the order $s$.

\begin{df}
The \textit{graded Euler number} of a
cochain complex
$(C^{\star}, d^{\star})$
is
defined as 
$Euler_i(C):= \sum_i (-1)^i \cdot q\mathrm{dim}(H^i).$
Also, the \textit{weighted graded Euler number} corresponding 
to a function $f$ of $(C^{\star}, d^{\star})$
is defined as 
$wEuler(C)_i:= \sum_i (-1)^i \cdot fq\mathrm{dim}(H^i).$
\end{df}

By (1) of Lemma \ref{lemma7d},
the coboundary operator $d$ 
is degree preserving,
hence
the (weighted) graded Euler number is also 
equal to the alternative sum 
of the ($f$)$q$dim of cochain groups,
those are 
$Euler_i(C)= \sum_i (-1)^i q\mathrm{dim}(C^i),$
and
$wEuler_i(C)= \sum_i (-1)^i fq\mathrm{dim}(C^i).$

When $f$ is a discrete harmonic function
it follows from the definition of  
harmonic polynomial, we have the 
following relationship between 
harmonic polynomial and the chain
complex constructed.
This relationship gives a weighted categorification 
to harmonic chromatic polynomials.
\begin{thm}
\label{Thm:ChromaticCate}
Let  $G(s)=(V,E)$ be a labelled graph and 
$f$ a discrete
harmonic function on $E$, 
then
$$\chi_{f}(G(s))=\sum_i (-1)^{i+1} \cdot fq\mathrm{dim}(C^i(G(s))).	$$
Also,
$$\widehat{\chi}_{f}(G)=
\sum_{s} \sum_i (-1)^{i+1}  \cdot fq\mathrm{dim}(C^i(G(s))).$$
\end{thm}

\begin{ex}
We calculate the weighted 
chain complex and the weighted 
chromatic polynomial of graph $G$ as an example.
Let $G=(V,E)$
be the following graph, 
and 
$f:E\to \mathbb{R}$
be a discrete harmonic function with degree $1$ on $E$.
The $f$-values of each edges are 
$a_1, a_2, a_3,$ 
and $a_4$
as follows.
The label is denoted by $s$.

Thus the sequence of chromatic cochain group is
$$0\to C^0(G(s))\to C^1(G(s)) \to C^2(G(s))\to C^3(G(s)) \to C^4(G(s))\to 0. $$
By definition, 
\begin{align*}
C^0 &\cong M^{\otimes 4};\\
C^1 &\cong M^{\otimes 3}\oplus M^{\otimes 3} \oplus M^{\otimes 3} \oplus M^{\otimes 3};\\
C^2 &\cong M^{\otimes 2} \oplus M^{\otimes 2} \oplus M^{\otimes 2} \oplus M^{\otimes 2} \oplus M^{\otimes 2} \oplus M^{\otimes 2};\\
C^3 &\cong M^{\otimes 2} \oplus M \oplus M  \oplus M;\\
C^4 &\cong M.
\end{align*}

Then, we calculate the 
$fq$dim of $C^{i}(G(s))$ for
$i=0,1,2,3,4.$

\begin{align*}
fq\mathrm{dim}(C^0)&= 0\cdot (1+q)^4 =0 ;\\
fq\mathrm{dim}(C^1)&= a_1(1+q)^3+a_2(1+q)^3+a_3(1+q)^3+a_4(1+q)^3\\
 &=
(a_1+a_2+a_3+a_4)(1+q)^3=0 ;\\
fq\mathrm{dim}(C^2)&=(a_1+a_2)(1+q)^2
+(a_1+a_3)(1+q)^2+(a_1+a_4)(1+q)^2\\
&+(a_2+a_3)(1+q)^2+(a_2+a_4)(1+q)^2
+ (a_3+a_4)(1+q)^2\\
&=3(a_1+a_2+a_3+a_4)(1+q)^2=0;\\
fq\mathrm{dim}(C^3)&=(a_1+a_2+a_3)(1+q)^2
+(a_1+a_2+a_4)(1+q)\\
&+(a_1+a_3+a_4)(1+q)+(a_2+a_3+a_4)(1+q)\\
&=-a_4(1+q)^2+a_4(1+q);\\
fq\mathrm{dim}(C^4)&= (a_1+a_2+a_3+a_4)(1+q)^4 =0.
\end{align*}
Note that since $f$ is a discrete harmonic function, 
$(a_1+a_2+a_3+a_4)=0$, 
and $(a_1+a_2+a_3)=-a_4$.

We suppose that
$\lambda=1+q$,
then the weighted graded Euler number of $C^{\bullet}(G(s))$
is $a_4\lambda^2-a_4\lambda$,
which coincidences to the opposite of the
weighted chromatic polynomial $\chi_{f}(G(s))$  that is 
calculated in Example 3.1.

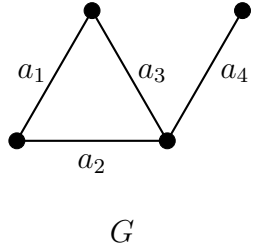
\begin{figure}[h]

	\begin{picture}(120,120)
		\put(10,10){
		\begin{tikzpicture}
			\draw[fill=black] (0,1.732) circle (3pt);
			\draw[fill=black] (1,0) circle (3pt);
			\draw[fill=black] (-1,0) circle (3pt);
			\draw[fill=black] (2,1.732) circle (3pt);
			\draw[thick] (1,0) -- (-1,0) -- (0,1.732)-- (1,0);
			\draw[thick] (1,0)-- (2,1.732);
			\node at (0,-0.3) {$a_2$};			
			\node at (-0.8,0.9) {$a_1$};
			\node at (0.8,0.9) {$a_3$};
			\node at (1.9,0.9) {$a_4$};
			
			\node at (0.4,-1.2) {$G$};
		\end{tikzpicture}}
		
		\end{picture}
		
\caption{An example of harmonic chromatic chain complex}		
\end{figure}

\end{ex}
\subsection{A categorification of weighted Tutte polynomials}

\label{Sec:CateHarmTutte}

In this subsection, we apply the method in \cite{categorification2}
to give a categorification to the 
Harmonic Tutte Polynomial defined in Section \ref{Sec:WeightedTutte}.
Similar to the last subsection, we give a  construction of a cochain complex
and calculate its weighted alternative sum to categorify the 
harmonic Tutte polynomial.
We omit some proofs in \cite{categorification2}.
We use the notation $TC^{\bullet}(G)$ and boundary operator
$td^{\bullet}$ in this section 
to distinguish the notations in this subsection and the last subsection. 

The cochain group
corresponding to the harmonic 
Tutte polynomial 
 is constructed by a 
 $\mathbb{Z}$-\textit{bigraded module}.
A $\mathbb{Z}$-\textit{bigraded module} 
is a module $\mathcal{A}$ 
consisting of a decomposition 
$\mathcal{A}=\bigoplus_{(i,j)\in \mathbb{Z}\oplus \mathbb{Z} } A_{i,j}. $ 
The graded dimension of $\mathcal{A}$ is defined as a
two-variable power series 
$q\mathrm{dim} \mathcal{A}:= \sum_{i,j} \dim_{\mathbb{Q}}(A_i \otimes \mathbb{Q}).$

Let $A$ and $B$ be polynomial rings and
 $A=\mathbb{Z}[x]/(x^2)$,
 $B=\mathbb{Z}[y]/(y^2)$,
where deg $x=(1,0)$, 
deg $y=(0,1)$.
The \textit{degree} 
of a polynomial 
is the largest natural number 
such that
the coefficient is not zero.
A simple calculation implies that
$q\mathrm{dim}A= 1+x$,
$q\mathrm{dim}B=1+y$,
and
$q\mathrm{dim}(A^{\otimes m} \otimes B^{\otimes n})=(1+x)^m (1+y)^n,$
for any $n,m\in \mathbb{N}$.

Let $G(s)= (V,E)$ be a labelled graph, 
and $f: 2^{\Omega} \to \mathbb{R}$
be a discrete function of degree~$d$.
Similar to the construction in the last subsection, 
we represent any subgraph $G'$ of $G$ whose vertices are 
$V$ by the edge vector.
We define the cochain group 
as following.

\begin{df}
The $q$-th \textit{Tutte chain group} of $G(s)$ is 
defined as
$$TC^{q} (G(s)):= \bigoplus_{l(\epsilon)=q}  A^{\otimes k(G_\epsilon)} \otimes 
B^{\otimes \beta_1(G_\epsilon)},  $$
where 
$\beta_1(G_\epsilon)$ denotes the one-dimensional Betti number of graph
$(G_\epsilon)$.

\end{df}

Next, we define a Tutte differential operator 
$td_q: TC^q(G) \to TC^{q+1}(G)$.
We begin with a 
map from $\mathcal{E}^a(G)$ to $\mathcal{E}^{a+1}(G)$
that is identity on all elements but 
change exact one $0$ to $1$.
That is, the map adds an edge into some subset $D\subset E$.
Known as a basic fact of 
algebraic topology, 
adding an edge $\tilde{e}$ to  $D\subset E$ 
leads to only two cases,
those are,
\begin{enumerate}
\item  $\tilde{e}$ joins two connected components of 
$G_D$;

\item $\tilde{e}$ forms a cycle in 
$G_D$.
\end{enumerate}
In the first case, 
$k(G_{D\cup  \tilde{e}})= k(G_D) -1 $,
and 
$\beta_1(G_{D\cup  \tilde{e}})= \beta_1(G_D).$
In the second case,
$k(G_{D\cup  \tilde{e}})= k(G_D) $,
and 
$\beta_1(G_{D\cup  \tilde{e}})= \beta_1(G_D) +1.$

For the first case, 
we define 
$td_{\epsilon}^A: A^{\epsilon_1}(G) \to A^{\epsilon_2}(G)$
to be the identity on all tensor factors other than those 
two tensor factors that are induced by two connected components
joined.
For the two joined tensor factors, 
we map $a_1\otimes a_2\in A\otimes A $ to
$a_1 a_2\in A$.
We also define
$td_{\epsilon}^B: B^{\epsilon_1}(G) \to B^{\epsilon_2}(G)$
to be the identity map.

For the second case,
we define
$td_{\epsilon}^A: A^{\epsilon_1}(G) \to A^{\epsilon_2}(G)$
be the identity map,
and
$td_{\epsilon}^B: B^{\epsilon_1}(G) \to B^{\epsilon_2}(G)$
be the 
module homomorphism mapping $b\in B^{\epsilon_1}(G) \mapsto b\otimes 1\in B^{\epsilon_2}(G)=B^{\epsilon_1}(G) \otimes B.$

Now by combining $td_{\epsilon}^A$
and $td_{\epsilon}^B$ in two different cases 
we can define the differential operator. 
More precisely,
\begin{df}
The differential operator 
$$td^{q}:TC^{q} (G(s)) \to TC^{q+1}(G(s))$$
is defined as 
$td^{q}= \sum_{l(\epsilon)=q} (-1)^{n(\epsilon)} ( td_{\epsilon}^A \otimes td_{\epsilon}^B)$,
where
$ n(\epsilon)=\sum_{i=1}^{k_0} \epsilon_i $,
and $k_0$ is the index of the new-added edge in 
the given order.

\end{df}

\begin{lem}
The following statements are true.
\begin{enumerate}
\item $td\circ td=0.$
\item $td$ is degree preserving.
\end{enumerate}

\end{lem}
The proof of $td\circ td=0$
can be found in \cite{categorification2}.
Since  $td\circ td=0$,
$\{TC^{\bullet}(G), td^{\bullet} \}$
is a cochain complex.
We denote by 
$TH^{\bullet}(G)$
for its cohomology.
It is shown in \cite{categorification2}
that, 
the  cochain complex does not 
depend on the choice of the order $s$.
The following 
theorem implies that the
constructed cochain 
complex combined with a 
harmonic function is 
a weighted categorification to 
harmonic Tutte polynomial.

\begin{thm}
\label{Thm:TutteCate}
Let  $G=(V,E)$ be a graph with an order $s$ and 
$f$ a discrete
harmonic function on $E$, 
then
$$T_{f}(M_G(s);x,y) = \sum_i (-1)^{i+1} \cdot fq\mathrm{dim}(C^i(G(s))).$$
Also,
$$\widehat {T}_{f}(M_G;x,y)=\sum_s \sum_i (-1)^{i+1} \cdot fq\mathrm{dim}(C^i(G(s))).$$
\end{thm}



\section*{Statements and Declarations}

\subsection*{Funding}

The second author is supported by JSPS KAKENHI (18K03217).

\subsection*{Competing Interests}

The authors have no affiliation with any organization with a direct or indirect financial interest in the subject matter discussed in the manuscript

\subsection*{Author Contributions}

All authors have participated in (a) conception and design, or analysis and interpretation of the data; (b) drafting the article or revising it critically for important intellectual content; and (c) approval of the final version.



\section*{Data availability statement}

The data that support the findings of this study are available from the corresponding author.

\end{document}